\documentclass{amsart}
\usepackage{amsmath,amssymb,dsfont,bm,here}
\usepackage{here,subfig,graphicx,color}

\newtheorem{defn}{Definition}[section]
\newtheorem{theorem}{Theorem}[section]
\newtheorem{lem}[theorem]{Lemma}
\newtheorem{proposition}[theorem]{Proposition}

\newtheorem{rem}{Remark}

 \newcommand{\la}{\lambda}
\newcommand{\zu}{[0,1]}

\newcommand\PCPS{\textit{Mathematical Proceedings of the Cambridge Philosophical Society} }

\newcommand\SPA{\textit{Stochastic processes and their applications} }
\newcommand\ep{\varepsilon}

\newcommand\RR{\mathbb{R}}
\newcommand\R{\mathbb{R}}
\newcommand\N{\mathbb{N}}
\newcommand\dimh{\dim_{\! H}}
\newcommand\ii{\mathcal{I}}

\newcommand{\card}{\mbox{Card}}

\begin{document}

\title[Multifractal analysis of sums of random pulses]
  {Multifractal analysis of sums of random pulses}

\author[Guillaume Sa\"es, St\'ephane Seuret]{Guillaume Saes, St\'ephane Seuret }

\address{Guillaume Sa\"es, Laboratoire d'analyse et de math\'ematiques appliqu\'ees \\
Universit\'e Paris-Est Cr\'eteil, Cr\'eteil, France \and Universit\'e de Mons,
Institut de Math\'ematique,
Le Pentagone,
7000 Mons,
Belgique}

\address{St\'ephane Seuret, Laboratoire d'analyse et de math\'ematiques appliqu\'ees \\
Universit\'e Paris-Est Cr\'eteil, Cr\'eteil, France}


\maketitle

\begin{abstract}
In this paper, we determine the almost sure multifractal spectrum of  a class of random functions constructed as sums of pulses with random dilations and translations.  In addition,    the continuity modulii of these functions is investigated.
\end{abstract}

\section{Introduction}\label{sec:1}

Multifractal analysis  aims at describing those  functions or measures whose pointwise regularity varies rapidly from one point to another. Such behaviors are   commonly encountered in various mathematical fields, from harmonic and Fourier analysis (reference) to stochastic processes and dynamical systems \cite{ref7,ref29,ref16,ref24,ref9,martinjaffard}. Multifractality is actually a typical property in many function spaces \cite{Bayart,BN,BS,JaffFrisch}.  Multifractal behaviors are also identified on real-data signals coming from turbulence, image analysis, geophysics for instance \cite{ref27,ref28,AJF}.   To quantify such an erratic   behavior  for a function $f\in L_{loc}^{\infty} (\RR)$,  it is classically called for the notion of  pointwise  H\"older exponent defined in the following way.

\begin{defn}
Let $f\in L^{\infty}_{loc}(\mathbb{R})$, $x_0 \in \R$  and $\alpha \geq 0$.
A function $f$ belongs to $C^\alpha (x_0)$ when there exist a polynomial $P_{x_0}$ of degree less than $\lfloor\alpha\rfloor$ and $K_\alpha\in\mathbb{R}_+^*$ such that 
$$\exists r\in\mathbb{R}_+^*,\ \forall x\in B(x_0,r), \ |f(x)-P_{x_0}(x-x_0)| \leq K_\alpha | x-x_0 |^\alpha .$$
The pointwise H\"older exponent of $f$ at a point $x_0$ is defined by
$$h_f (x_0)=\sup\{\alpha \geq 0 :\ f\in C^\alpha (x_0)\} .$$
\end{defn}
 
 In order to  describe globally  the pointwise behavior of  a given function of a process,   let us introduce the following iso-H\"older  sets.

\begin{defn}
Let $f\in L^{\infty}_{loc}(\mathbb{R})$ and $h \geq 0$. The iso-H\"older set $E_f (h) $ is  $$E_f (h) =\{x\in\mathbb{R} : \ h_f (x)=h\}.$$
\end{defn}

The functions studied later  in this paper have fractal, everywhere dense,  iso-H\"older sets. It is therefore relevant to call for the Hausdorff dimension, denoted by $\dimh $, to distinguish them, leading to the notion of  multifractal spectrum.

\begin{defn}
 The multifractal spectrum of $f\in L^{\infty}_{loc}(\mathbb{R})$ on a Borel set $A\subset \R$ is the mapping defined by
$$D^A_f : \left\{
\begin{array}{lcl}
\mathbb{R}_+ & \longrightarrow & \mathbb{R}_+ \cup\{-\infty\} \\
h & \longmapsto & \dimh  (E_h \cap A).
\end{array}
\right. $$
\end{defn}

By convention, $\dimh (\emptyset)=-\infty$. The multifractal spectrum of a function or a stochastic  process $f$ provides one with a global information on the geometric distribution of the singularities of $f$.

In this article, we aim at computing the multifractal spectrum of a class of    stochastic processes   consisting in sums of    dilated-translated versions of a function (referred to as a "pulse") that can have an arbitrary form.
The translation and dilation parameters are random in our context.  The present article  hence follows a longstanding research line consisting in studying    the regularity properties  of (irregular) stochastic processes  that can be built by an additive construction, including for instance additive L\'evy processes,  random sums and wavelet series, random tesselations, see  \cite{jaffard1999levy,Xiao,ref9,Gousseau,Calka} amongst many references.

 Our model is particularly  connected to other models previously introduced and studied by many authors.  

For instance,  in  \cite{ref1}   Lovejoy and  Mandelbrot   modeled rain fields by a 2-dimensio\-nal  sum of random pulses  constructed as follows.  Consider a random Poisson measure   $N$ on $E =\mathbb {R}_+^* \times \mathbb{R}_+^* \times \mathbb{R}^d$, as well as  a "father pulse"  $\psi:\mathbb{R}^d \to \mathbb{R}$, $\alpha\in ]0,2[$ and $\eta\in ]0,1]$.  Lovejoy and   Mandelbrot  built and studied the process   $M:\mathbb{R}^d \to \mathbb{R}$ defined by
\begin{equation}\label{eqF}
M(x)=\int_{(\la,w,\tau)\in E} \lambda^{-\alpha} \psi (w^{\frac{1}{\eta}} (x-\tau)) N(d\la,dw,d\tau) = \sum_{(\lambda,w,\tau)\in S} \lambda^{-\alpha}   \psi_{\lambda, w ,\tau}(x),    
\end{equation}
where $S$ is the set of random points induced by the Poisson measure and $\psi_{\lambda,w,\tau}(x) : =  \psi (w^{\frac{1}{\eta}} (x-\tau))$  and $\eta=1$. 
 
In \cite{ref3},  Cioczeck-Georges and Mandelbrot showed that negatively correlated fractional Brownian motions ($0<H<1/2$) can be obtained as a limit (in the sense of distributions) of a sequence of processes defined as in (\ref{eqF}) with $\psi$ a well-chosen  jump function, $\alpha\in ]0,2[$ and $\eta=1$.  
Later, in  \cite{ref4}, the same authors  proved   that any fractional Brownian motion with Hurst index $H\in (0,1)\setminus\{1/2\}$ is a   limit  of a sequence of processes $\{M_n (x), x\geq 0 \}_{n \in\mathbb{N}}$ defined as in (\ref{eqF}) with    $\psi$    a conical   or semi-conical function. Other versions with general   pulses $\psi$ have been investigated in   \cite{ref15}.

\begin{figure} 

\includegraphics[scale=0.38]{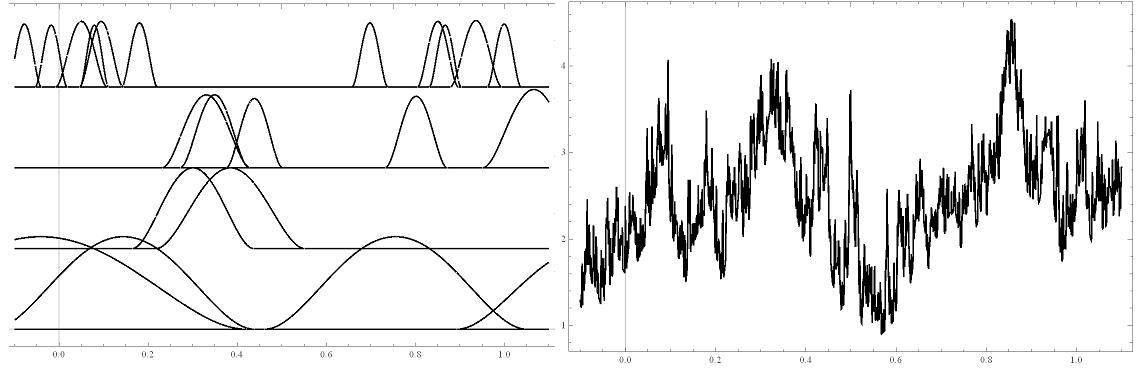}
\includegraphics[scale=0.38]{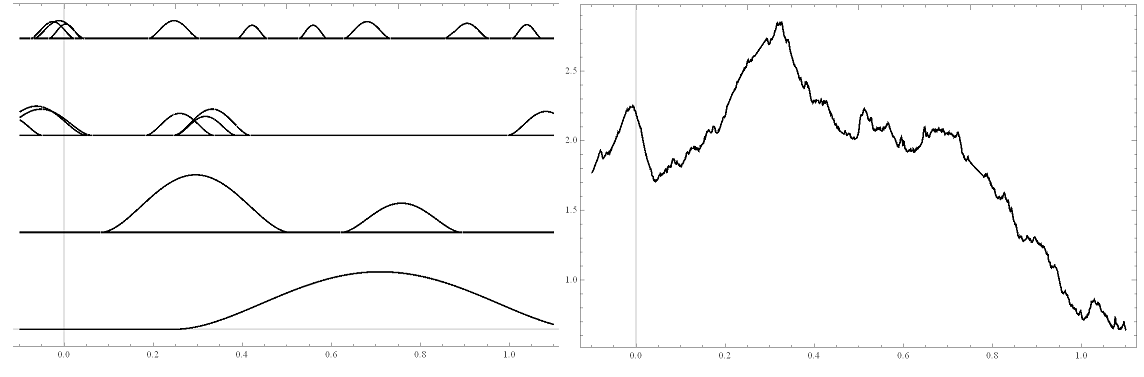}
 \caption{Two sample paths obtained with different pulses and parameters } \label{fig3}
\end{figure}

In   \cite{ref5}, Demichel studied  a model in which  only the    position coefficients $(X_n)_{n\geq 1}$ are random  : the corresponding model is  written
\begin{equation}\label{eq1.8}
    G(x)=\sum_{n=1}^{+\infty} a_n \psi(\lambda_n^{-1} (x-X_n)), \ \ \ x\in\mathbb{R} 
\end{equation}
where $(a_n)_{n\in\mathbb{N}^*}$ and $(\lambda_n)_{n\in\mathbb{N}^*}$ are two deterministic positive  sequences such that $\sum\limits_{n\in\mathbb{N}^*}a_n=+\infty $   and $(\lambda_n)_{n\in\mathbb{N}^*}$ is decreasing to $0$, and $X_n\sim U([0,1])$ is an i.i.d. sequence of random variables. 
The same example is developed   in \cite{ref10, ref21} where Demichel,  Falconer and   Tricot impose   that $ a_n=n^{-\alpha}$   with $0<\alpha<1$,    $\lambda_n = n^{-1}$, and  $\psi:\mathbb{R}\rightarrow\mathbb{R}$ is an even, positive continuous function,  decreasing  on $[0,1]$, equal to 0 on $[1,+\infty[$ satisfying   $\psi (0)=1$. \\Calling $\Gamma_G$ the graph of the process $G$ and $\dim_B \Gamma_G$ its box-dimension, they showed that as soon as there exists an interval $I$ on which  $\psi \in \mathcal{C}^H(\R)$ (the space of   global H\"older real functions of exponent $H$)  and is $\mathcal{C}^ 1$-diffeomorphic on some interval,   then almost surely
\begin{equation}\label{eq1.7}
2-\alpha \leq \dimh  (\Gamma_G)\leq \dim_B (\Gamma_G) \leq 2-\min\{\alpha,h\} . 
\end{equation}
 See also \cite{ref14} for the box dimension of $\Gamma_G$, or \cite{Roueff,Xiao2} for a more systematic study of graph dimensions.
When $\alpha <h$,  almost surely $\dimh  (\Gamma_G) = \dim_B (\Gamma_G) = 2- \alpha$.
In \cite{ref13}, Ben Abid    gave alternative conditions for the convergence of such processes $G$, also determining the uniform regularity of $G$, i.e. to which global H\"older space $\mathcal{C}^H(\R)$ $G$ may belong to, almost surely.

\medskip

Our purpose is to study another, somehow richer,   model of sums of  random pulses.

\section{A  model with additional randomness} \label{sec:2}

The  stochastic processes $F$   considered in this article are natural extensions of the previous models, and fit   in the general study of pointwise regularity properties of rough sample paths of stochastic processes. As  in the aforementioned works, we obtain  results regarding their  existence and  regularity properties. We go further by providing a complete multifractal analysis  of $F$ and by investigating  various modulii of     continuity.

Fix a probability space $(\Omega, \mathcal{F},\mathbb{P})$ on which all random variables and stochastic processes are defined.

Let  $(C_n)_{n\in\mathbb{N}^*}$  be a point  Poisson process whose intensity  is the Lebesgue measure on  $\mathbb{R}_+$, and let   $S$ be another point Poisson process,  independent with  $(C_n)_{n\in \N^*}$,  whose intensity  is  the Lebesgue measure on $\mathbb{R}_+^* \times [0,1]$. We write $S=(B_n,X_n)_{n\in\mathbb{N}^*}$ where the sequence  $(B_n)_{n\in \N^*}$ is increasing. By construction, the three sequences of random variables $(C_n)$, $(B_n)$ and $(X_n)$ are mutually independent.
 
\begin{defn}
\label{defF}
Let $\psi:\mathbb{R}\rightarrow\mathbb{R}$  be a Lipschitz   function  with support equal to $[-1,1]$, $\alpha\in (0,1)$ and $\eta\in (0,1)$.  The (random) sum of   pulses  $F$ is defined  by
\begin{equation}\label{eq2.1}
F(x)=\sum_{n = 1}^{+\infty} C_n^{-\alpha}  \psi_n(x), \ \ \mbox{ where } \ \  \psi_n(x):= \psi (B_n^{\frac{1}{\eta}} (x-X_n)) 
\end{equation}
\end{defn}

The parameter   $\alpha$ will be interpreted as  a regularity coefficient,  and $\eta$ as a lacunarity coefficient.  Observe that the support of $\psi_n$ is the ball $B(X_n, B_n^{-1/\eta})$ ($B(t,s)$ stands for the ball with centre $t$, radius $s$).

 The stochastic process $F$ possesses interesting properties on the interval $[0,1]$ only, since   $X_n\in [0,1]$. However, this is not a restriction at all, since $F$ can easily be extended to $\R$ as follows.
 
For every integer $m$, consider $F_m$, an independent copy of $F$ but shifted by $m$.  Then, 
$$\widetilde F :=\sum_{m\in \mathbb{Z}} F_m $$
   enjoys the same  pointwise   regularity properties as $F$. It is interesting to see that this new process  $\widetilde{F}$ has now stationary increments, and enlarges the quite narrow   class of stochastic processes with stationary increments whose multifractal analysis is completely understood.

\begin{figure}
\centering
\includegraphics[scale=0.3]{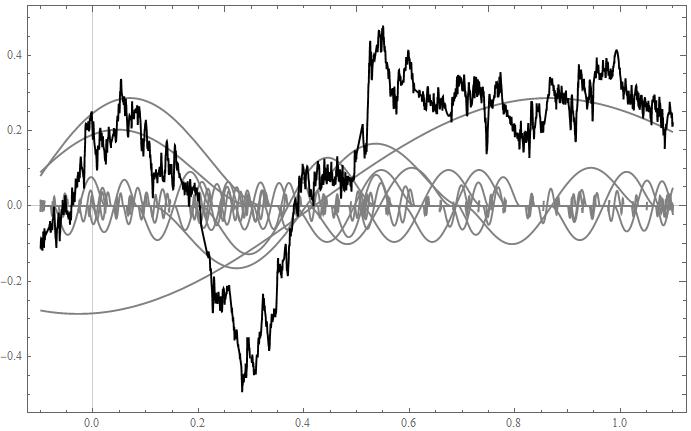}
\caption{Sample path of $F$ computed with   1000  dilated and translated pulses}  
\label{fig4}
\end{figure}

In \cite{ref9}, using for $\psi$ a smooth wavelet generating an orthonormal basis,    S. Jaffard studied the lacunary random wavelet series  
$$
W(x)=\sum_{j\in\mathbb{N}} \sum_{k\in\mathbb{Z}} C_{j,k}2^{-j\alpha}  \psi_{j,k}(x), \ \ \ x\in\mathbb{R} ,
$$
where for all $(j, k)\in\mathbb{N} \times \mathbb{Z}$,   $\psi_{j,k} (x)=\psi (2^j x-k)$ and the wavelet coefficients $C_{j, k}$ are independent random variables wavelets whose law is a Bernoulli measure with parameter $2^{-j\eta}$ (hence, depending  on $j$ only).  The main difference  between the lacunary wavelet series and our model (motivating our work) is that  not only  dilations $(B_n)_{n\in\mathbb{N}^*}$ but also the    translations $(X_n)_{n\in\mathbb{N}^*}$    are random in our case. Hence our  interest in   $F$ (and in $\widetilde F$) comes from the fact that it is not based on a   dyadic grid, hence providing  one with a   homogeneous model  more natural from a probabilistic point of view, the process  $\widetilde F$ having stationary increments.  The main results of this article  concern the global and pointwise regularity properties of $F$, which are proved to be similar to those of $W$.

\medskip

We start by the multifractal properties of $F$.
\begin{theorem}
\label{maintheorem} 
  Let $F$ be  as in Definition \ref{defF}. 
  With probability one, one has
$$
    D^{[0,1]} _F (H)=\left\{
    \begin{array}{lcr}
         \frac{H}{\alpha} & \mbox{if} & H\in [\alpha\eta, \alpha] , \\
         -\infty & \mbox{ else.} 
    \end{array}
    \right.
$$
\end{theorem}

The other results concern   the almost-sure global regularity of $F$ and its modulii of continuity. Let us recall the notions of modulus of continuity.

\begin{defn}
A non-zero increasing mapping  $\theta:\R^+\to \R $ is a    modulus  of continuity when it satisfies
\begin{enumerate}
\item $\theta (0)=0$,
\item There exists $K>0$ such that  for every $h\geq 0$, $\theta(2h)\leq K \theta(h)$.
\end{enumerate}
\end{defn}
Function spaces are naturally associated with  modulii of continuity.
 \begin{defn}
 \label{deftheta}
A function $f\in L^\infty _{loc}(\mathbb{R})$ has   $\theta:\R^+\to \R $ as  uniform   modulus  of continuitywhen   there exists $K>0$ such that
\begin{equation}
\forall h\in\mathbb{R}_+, \ w_f (h):=\sup_{|x-y|\leq h} |f(x)-f(y)|\leq K \theta(h) .  \nonumber
\end{equation}

A function $f\in L^\infty _{loc}(\mathbb{R})$ has   $\theta:\R^+\to \R $ as local      modulus of continuity at $x_0\in \R$  when   there exist $\eta_{x_0}>0$ and  $K_{x_0}>0$ such that
\begin{equation}
\label{localmodulus}
\forall  x \mbox{ such that } |x-x_0|\leq \eta_{x_0}, \ \ \ \    |f(x)-f(x_0)|\leq K_{x_0} \theta(|x-x_0|) .   
\end{equation}

A function $f\in L^\infty _{loc}(\mathbb{R})$ has   $\theta:\R^+\to \R $ as  almost-everywhere      modulus   of continuity  when $\theta$ is a local   modulus of continuity for $f$ at Lebesgue almost every $x_0\in \R$. 
\end{defn}

When $\alpha\in (0,1)$ and $\theta(h)= \theta_\alpha(h):= |h|^\alpha$, the functions having $\theta_\alpha$ as uniform   modulus of continuity is exactly the set $\mathcal{C}^\alpha(\R)$ of $\alpha$-H\"older functions (to deal with exponents $\alpha \geq 1$, the definition of $w_f (h)$ must be modified and use  finite differences of higher order).

\medskip

Our result theorem regarding continuity moduli is the following.

\begin{theorem}
\label{mainth2}
  Let $F$ be  as in Definition \ref{defF}.    With probability 1:
  \begin{enumerate}
  \item
 Tthe mapping  $h\mapsto   |h|^{\alpha\eta} |\log_2 (h)|^{2+\alpha}$ is a uniform    modulus of continuity of  $F$. 
\item
 The mapping  $h\mapsto  |h|^{\alpha } |\log_2 (h)|^{2+\alpha}$ is an almost everywhere    modulus of continuity of  $F$.
 \item
 At Lebesgue almost every $x_0\in \zu$, the local    modulus of continuity of $F$ at $x_0$ is  larger than  $h\mapsto  |h|^{\alpha } |\log_2 (h)|^{2\alpha}$.

\end{enumerate}
\end{theorem}

\begin{rem}
Items (ii) and (iii) above provide us with  a tight window for the optimal almost everywhere    modulus of continuity $\theta_F$ of  $F$, i.e. 
$$  |h|^{\alpha } |\log_2 (h)|^{2\alpha} \leq \theta_F(h) \leq |h|^{\alpha } |\log_2 (h)|^{2+\alpha}.$$
The investigation of a sharper estimate for this   modulus of continuity is certainly of interest. For instance, S. Jaffard was able to obtain a precise characterization in the case of lacunary wavelet series, see Theorem 2.2 of \cite{ref9}.
\end{rem}

\begin{rem}
The result can certainly be extended to  dimension $d>1$ with parameters $\alpha>1$, provided that $\psi \in  C^{\lfloor \alpha \rfloor +1}(\mathbb{R}^d)$. This would add technicalities   not developed here. 
\end{rem} 

\begin{figure}
\includegraphics[scale=0.4]{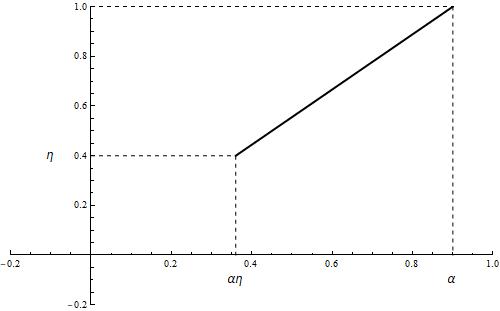}
\caption{Multifractal spectrum in the case $\alpha=0.9$ and $\eta=0.4$}
\label{fig7}
\end{figure}

The paper is organized as follows. Preliminary results are given in Sections \ref{sec:3} and \ref{sec:5}. A key point will be to estimate  for $j\in\mathbb{N}$,   the maximal number of integers   $n\in\mathbb{N}^*$ satisfying $2^{j} \leq B_n^{\frac{1}{\eta}}<2^{j+1}$, such that the support of $\psi_n$ contains a given point  $x\in [0,1]$ (a bound uniform in $x\in [0,1]$ is obtained).   More specifically, we will focus on the so-called "isolated" pulses $\psi_n$, i.e. those pulses whose support intersect only a few number of supports of other pulses with comparable support size.   These random covering questions are dealt with in Section \ref{sec:5}. This is key   to obtain lower and upper estimates for the pointwise   H\"older exponents of $F$ at all points and to get Theorem \ref{maintheorem}. More precisely, in Section \ref{sec:4},   item (i) of Theorem \ref{mainth2} is proved, and   a uniform lower bound for all  the pointwise H\"older exponents of $F$ is obtained. In   Sections \ref{sec:6} and Section \ref{sec:7}, we relate the approximation rate of a point $x\in \zu$ by some family of random balls to the pointwise regularity of $F$. This allows us to derive the almost sure multifractal spectrum of $F$ in Section  \ref{sec:8}. 
In Section \ref{sec:9}, we explain how to get the almost everywhere     modulus of continuity for $F$ (items (ii) and (iii) of Theorem \ref{mainth2}). Finally,  Section \ref{sec:10} proposes some research perspectives.

\section{Preliminary results} \label{sec:3}

Preliminary results are exposed, some of which can be found in standard books  \cite{ref6, ref23}.  

For   $j\in\mathbb{N}$,  define 
\begin{eqnarray}
A_0 &=& \{n\in\mathbb{N}^* : \ 0 \leq B_n^{\frac{1}{\eta}} \leq 1\} ,\nonumber\\
A_j &=& \{n\in\mathbb{N}^* : \ 2^{j-1} < B_n^{\frac{1}{\eta}} \leq 2^j\} \ \ \ \mbox{ when  } j> 0 ,\label{eq3.4}\\
N_j &=&\card(A_j) . \nonumber
\end{eqnarray}

From its definition,  each  $N_j$ is a Poisson random variable with parameter $2^{\eta j}-2^{\eta(j-1)}$. 


\begin{lem}\label{lem3.3}
Almost surely, there exist for $j$ large enough,
\begin{equation}\label{eq3.5}
2^{\eta j(1-\varepsilon_j)}\leq N_j \leq 2^{\eta j(1+\varepsilon_j)} \ \ \mbox{ with } \ \   \varepsilon_j=\frac{\log_2 (j)}{\eta j} .
\end{equation}
\end{lem}

\begin{proof}
Introduce the counting random function $(M_t)_{t\in\mathbb{R}_+^*}$ of the point process $(B_n)_{\mathbb{N}^*}$ as $M_t=\sup\{ n\in\mathbb{N}^* \ : \ B_n\leq t \}=\sum_{n\in\mathbb{N}^*}  \mathds{1}_{B_n\leq t}$.

For all $0\leq s < t$,  $M_t-M_s$ is a Poisson variable with parameter $(t-s)$. Noting that   $N_j=M_{2^{\eta j}}-M_{2^{\eta(j-1)}}$,   the random variable $ N_j $ has a Poisson distribution of parameter $  a2^{\eta(j-1)}$ where $a=2^\eta-1$. By the Bienaym\'e-Tchebychev inequality, since $\mathbb{E}[N_j]=a 2^{\eta (j-1)}$, one has   
\begin{equation}\label{eq3.7}
\mathbb{P}(|N_j - a2^{\eta(j-1)} |\geq j 2^{\frac{\eta}{2}(j-1)})\leq \frac{a 2^{\eta (j-1)}}{j^2 2^{\eta (j-1)}} \leq \frac{a}{j^2} .   
\end{equation}
By the Borel-Cantelli lemma, a.s. for $j$ large enough, $|N_j-a 2^{\eta(j-1)}|\leq j 2^{\frac{\eta}{2}(j-1)}$.
In particular,  for every $\alpha>0$ and $j$ large enough, $j^{-\alpha }2^{j\eta}\leq N_j\leq j^\alpha 2^{j\eta}$. This  implies \eqref{eq3.5}.  
\end{proof}

From (\ref{eq3.7}), for every $\alpha>0$,  there exists $K> 0$ such that for every $j$,
\begin{equation}\label{eq3.8}
\mathbb{P}(N_j\notin [2^{\eta j (1-\alpha \varepsilon_j)},2^{\eta j (1+\alpha \varepsilon_j)}])\leq \frac{K}{j^2}.
\end{equation}
Observe that \eqref{eq3.8} indeed holds for every $j$ with a suitable choice for $K$.
This will be used later. Bounds for the  random variables $B_n$ and $C_n$ are deduced from the previous results.


\begin{lem}\label{lem3}
Almost surely,   for all $j\in\mathbb{N}$ large enough and $n\in A_j$,
\begin{eqnarray}
\label{eq3.9}
 2^{\eta j (1-\varepsilon_j)} \leq B_n, C_n \leq   2^{\eta j (1+\varepsilon_j)} .
\end{eqnarray}
\end{lem}

\begin{proof}
It is standard (from the law of large numbers for instance) that almost surely, for every $n\in\mathbb{N}^*$ large enough
\begin{eqnarray}
\label{eq3.2}
    \frac{n}{2}   \leq B_n \leq 2 n \ \ \mbox{ and } \ \    \frac{n}{2}   \leq C_n \leq 2 n .
\end{eqnarray}
  
 Let $J\in\mathbb{N}$ be large enough so that \eqref{eq3.5} holds for $j\geq J$. Call $A=  \sum_{j'=0}^{J } N_{j'}   $.
 
 Let  {$j\geq J+1$}, and $n\in A_j$. By definition, one has $\sum_{j'=0}^{j-1} N_{j'} \leq n \leq \sum_{j'=0}^j N_{j'}$. 
 
We apply by \eqref{eq3.7} with $\alpha=1/2$.  On one side,    
$$ \sum_{j'=0}^{j -1} N_{j'} \geq N_{j-1} \geq 2^{\eta (j-1)(1-\alpha \varepsilon_{j-1})} \geq K_1 2^{\eta j (1-\alpha \ep_j)} \geq 2^{\eta j (1-  \ep_j)} .$$
 On the other side,  since $j\ep_j$ is increasing with $j$, when $j$ becomes large one has
\begin{eqnarray*}
 \sum_{j'=0}^j N_{j'} \leq  A +  \sum_{j'=J+1 }^{j} 2^{\eta j'(1+\alpha \varepsilon_{j'})} \leq  A + 2^{\eta j \ep_j}  \sum_{j'=J+1 }^{j } 2^{\eta j'}\leq K_2 2^{\eta j (1+\alpha \varepsilon_{j})}  ,\end{eqnarray*} 
since $A$ is finite. The last term is less than $  2^{\eta j (1+  \varepsilon_{j})} $, so combining this  with  \eqref{eq3.2}  gives \eqref{eq3.9}.
\end{proof}


Finally, for all $j\in\mathbb{N}$ and $n\in A_j $, additional information on the number of pulses $\psi_n $ for $n\in A_j$  (see  \eqref{eq2.1}) whose support contains a given $x\in [0,1]$ is needed. So, for $x\in [0,1]$, $r>0$ and $n \in\mathbb{N}^*$, set $$
 T_n(x,r )=\left\{ \begin{array}{ll}
1 & \mbox{ if } B(X_n, B_n^{-\frac{1}{\eta}})\bigcap B(x,r)\neq \emptyset, \\
0 & \mbox{ otherwise.}
\end{array}\right. 
$$

 Next Lemma  describes the number of overlaps between the balls  $B(X_n, B_n^{-\frac{1}{\eta}}) $ for $n\in A_j$.  It is an improvement of some properties  proved in  \cite{ref5}.

\begin{lem}\label{lem3.6}
Almost surely, there exists $K>0$ such that for every $x \in [0,1]$, for every  $J,j \in\mathbb{N}$,        
\begin{eqnarray}\label{eq3.11}
 \sum_{n\in A_j} T_n(x,2^{-\eta J})  \leq  K  j^2  \max(1,2^{\eta(j-J)}).
\end{eqnarray}\end{lem}
 
\begin{proof}
We first work on the dyadic grid. 
Let $j\in\mathbb{N}$ and $j_\eta=\lfloor\eta j\rfloor$. Observe that  $ [0,1]=\bigcup_{k=0}^{2^{j_\eta}-1} I_{j_\eta,k} $, where $ I_{j_\eta,k}=[k2^{-j_\eta},(k+1)2^{-j_\eta}].  
$
For $k\in\{0,1,...,2^{j_\eta}-1\}$,  and set
\begin{align}
\label{defljk}
L_{j,k} = \card\left\{ n\in A_j \ : \ X_n\in I_{j_\eta,k}\pm 2^{-j_\eta+1 } \right\}.
\end{align}

Let us estimate $p_j=\mathbb{P}(\exists k\in\{0,1,...,2^{j_\eta}-1\}: L_{j,k} > j^2).$
Using Bayes'formula,
\begin{align*}
p_j &= & \mathbb{P}(\exists k\in\{0,1,...,2^{j_\eta}-1\},  \ L_{j,k}> j^2 | N_j\in [2^{\eta j (1-\varepsilon_j)},2^{\eta j (1+\varepsilon_j)}])\\ &&\times \mathbb{P}(N_j\in [2^{\eta j (1-\varepsilon_j)},2^{\eta j (1+\varepsilon_j)}])  \\
&&+ \ \ \  \mathbb{P}(\exists k\in\{0,1,...,2^{j_\eta}-1\},  L_{j,k}>  j^2| N_j\notin[2^{\eta j (1-\varepsilon_j)},2^{\eta j (1+\varepsilon_j)}])\\ &&\times\mathbb{P}(N_j\notin[2^{\eta j (1-\varepsilon_j)},2^{\eta j (1+\varepsilon_j)}]).  
\end{align*}
Applying  (\ref{eq3.8}), there exists $K>0$ such that for every $j$,
\begin{equation}\label{eq3.13}
p_j \leq   \sum_{N\in  \{  \lfloor 2^{\eta j (1-\varepsilon_j)} \rfloor ,...,  \lfloor2^{\eta j (1+\varepsilon_j)}\rfloor\}} p_{j,N}  \mathbb{P}(N_j=N)\ +\frac{K}{j^2} , \end{equation}
where for every  integer $N$,
$
p_{j,N}=\mathbb{P}(\exists k\in \{0,1,...,2^{j_\eta}-1\} : L_{j,k} >j^2 |N_j=N).
$
Obviously, $p_{j,N}$ is increasing with $N$, hence   $p_{j} \leq  p_{j, 2^{\lfloor\eta j (1+\varepsilon_j)\rfloor}} +\frac{K}{j^2}$.

Conditioned on $N_j=n_0:=2^{\lfloor\eta j (1+\varepsilon_j \rfloor)} $, the law of  each $L_{j,k}$ is binomial $B(n_0,p)$  with parameters $n_0$ and $p=\mathbb{P}(X_n\in  I_{j_\eta,k}\pm 2^{-j_\eta+1} )$. 

Recall the argument by  Demichel and  Tricot used in Lemma 2.1 of    \cite{Demichel2}: For $Y\sim B(n_0,p)$, then  for every  $m\geq 1$,
\begin{equation*}
\mathbb{P}(Y>m) \leq \frac{(n_0 p)^m}{m!}.
\end{equation*}
In particular, in our case, since $p\leq 3\cdot 2 ^{-j_\eta} \leq 6\cdot 2^{-\eta j }$, one has
\begin{equation*}
\mathbb{P}(L_{j,k} > j^2 | N_j =n_0) \leq \frac{(n_0 p)^{j^2}}{(j^2)!} \leq \frac{( 6\cdot 2^{\eta j (1+   \varepsilon_j ) -\eta j} )^{j^2} }{(j^2)!} = \frac{ (6\cdot j)^{j^2} }{(j^2)!}
\end{equation*}
 Hence, 
$$p_{j,\lfloor 2^{ \eta j (1+\varepsilon_j)}\rfloor} \leq \sum_{k=0}^{2^{j_\eta -1}} \frac{(6\cdot j)^{j^2}}{(j^2)!} \leq \frac{2^{j_\eta}(6\cdot j)^{j^2}}{(j^2)!}.$$
Recallng \eqref{eq3.13}, one concludes that
$$p_j \leq \frac{2^{j_\eta}(6\cdot j)^{j^2} }{(j^2)!} + \frac{K}{j^2}$$
which is the general term of a convergent series.

  Borel-Cantelli lemma gives that almost surely, for all  $j\in\mathbb{N}$ large enough and for every   $k\in\{0,1,..,2^{j_\eta}-1\}$, $L_{j,k}\leq j^2$. So, almost surely, there exists   $K>0$ such that for every $j\geq 1$, for every  $k\in\{0,1,..,2^{j_\eta}-1\}$, $L_{j,k}\leq K j^2$.
 
\medskip

To conclude now,   fix an integer $J$, $x\in \zu$ and $2^{-J-1} \leq r \leq 2^{-J}$. 
Two cases are  distinguished:

\begin{itemize}

\smallskip\item
When $  j\leq J$:  calling again $j_\eta=\lfloor j\eta\rfloor$, the point $x$ belongs to  a unique interval $I_{j_\eta,k_x}$ (for some unique integer $k_x$).   When $n\in A_j$, observe that $T_n(x,2^{-\eta J}) =1$ if and only if $|X_n-x|\leq 2^{-\eta J} + B_n ^{-1/\eta}  \leq  2^{-\eta J} +2^{-j}$.  This may occur only when $X_n\in I_{j_\eta,k_x}\pm  (2^{-\eta J}+2^{-j} ) \subset  I_{j_\eta,k_x}\pm  2\cdot 2^{-j_\eta} $, since $ j   \le J$.

 From the consideration above,   there are at most $K j^2$ points $X_n$, $n\in A_j$, such that $T_n(x,2^{-\eta J}) =1$, hence \eqref{eq3.11}.

\smallskip\item
 When $j >  J$:   As above, when $n\in A_j$,     $T_n(x,2^{-\eta J}) =1$ may occur only if $|X_n-x|\leq 2^{-\eta J} + B_n ^{-1/\eta} \leq  2^{-\eta J} +2^{-j} \leq 2^{- \lfloor \eta J \rfloor +1}  $.      The interval $[x- 2^{- \lfloor \eta J \rfloor +1}, x+2^{- \lfloor \eta J \rfloor +1}]$ is covered by at most $\lfloor 2^{ \eta(j-J)+3}\lfloor$ intervals $I_{\lfloor j\eta\rfloor,k}$, and each of these intervals contain at most $K j^2$ points $X_n$. So, $T_n(x,2^{-\eta J}) =1$ for at most $Kj^2 2  ^{ \eta(j-J)+3}  2$ integers   $n\in A_j$. Hence the result  \eqref{eq3.11}.
\end{itemize} \end{proof}

  Observe that the degenerate case $J=+\infty$ also holds in this case, i.e.  almost surely, there exists $K>0$ such that for every $x \in [0,1]$, for every  $j \in\mathbb{N}$,     one has  
\begin{eqnarray}\label{eq3.11.bis}
\sum_{n\in A_j} T_n (x) =\sum_{n\in A_j} T_n (x,0)  \leq  K  j^2.
\end{eqnarray}

\section{Distribution of isolated pulses} \label{sec:5}

There may be several pulses $\psi_n$  with $n\in A_j$ whose support intersect  each other, creating unfortunate  irregularity compensation phenomena and making the estimation of local increments of the process $F$ difficult. In order to circumvent this issue,  the knowledge on the distribution of the $\psi_n$'s shall be improved.

For this, fix $\gamma \in [1,1/\eta]$ and   $p_0 \in \N$ so large that  
\begin{equation}
\label{choice-p}
p_0 >   \frac{3+3 \alpha}{1-\alpha\eta} .
\end{equation}
 Let us introduce  for any $j\in\mathbb{N}$ the sets
\begin{eqnarray}\label{eq5.1}
&\widetilde{A}_j=\bigcup_{j'=\lfloor (1-p\eta\varepsilon_j)j \rfloor}^{\lfloor \gamma j \rfloor} A_{j'} \ \ \ \mbox{ and }\ \ \ \widetilde{N}_j=\card(\widetilde{A}_j)\\\label{eq5.2}
&\ii _j=\{ n\in A_j \ : \ \forall m\in\widetilde{A}_j,\ n\neq m ,\ B(X_n,B_n^{-\frac{1}{\eta}})\cap B(X_m,B_m^{-\frac{1}{\eta}})=\varnothing \}
\end{eqnarray}
The elements of $\ii_j$ are integers $n\in A_j$ such that the support of $\psi_n$ does not intersect any support of $\psi_m$ for $m\in\widetilde{A}_j $ with $m\neq n$.

\begin{defn}\label{def9}
A point $X_n$ with $n\in \ii _j $ is called an isolated point.
\end{defn}

The distribution of the isolated points $\{X_n\}_{n\in \mathcal{I}_j }$ is further investigated. Indeed,  as said above,   such information is key  to obtain upper and lower bounds for the H\"older exponent of $F$ at any point $x$ (see Sections \ref{sec:6} and \ref{sec:7}).  To describe the distribution of $\{X_n\}_{n\in \mathcal{I}_j }$, consider the  two limsup sets 
\begin{eqnarray}
\label{defgdelta}G_\delta & = &\limsup_{j\rightarrow +\infty} \bigcup_{n\in A_j} B(X_n,B_n^{-\delta})  \\
\label{defg'delta} G'_\delta & = &\limsup_{j\rightarrow +\infty} \bigcup_{n\in \ii _j} B(X_n,B_n^{-\delta(1-\widetilde{\varepsilon}_j)}), \ \ \mbox{ where $\widetilde{\varepsilon}_j=\log_2 (16j \log_2 j)/(\eta j)$. }
\end{eqnarray}
\begin{rem}
\label{reminclusions}
Note that  as soon as $\delta >\delta'$, $G_\delta\subset G_{\delta ' }$ and $ G'_\delta\subset G'_{\delta ' }$.  
\end{rem}

In the next sections, it is proved that  $ G_\delta $ contains points whose pointwise  H\"older exponent of $F$ is lower-bounded by $\alpha /\delta$ and $G'_{\delta}$ points whose pointwise H\"older exponent of $F$ is upper-bounded by $\alpha /\delta$. The idea is that on the support of an isolated pulse, the process $F$ has large local oscillations, thus forming points around which $F$  possesses a low regularity. 
 
\begin{figure}
    \centering
    \includegraphics[scale=0.7]{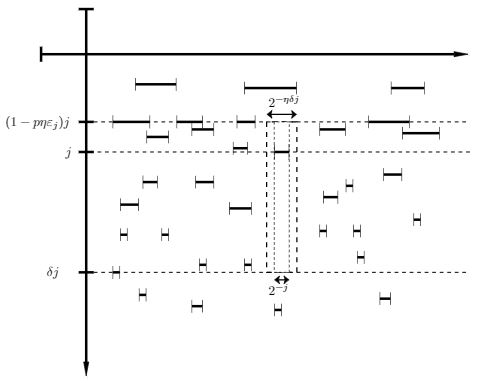}
    \caption{Representation of pulses supports in $\ii _j$.}
\end{figure}

It is a classical result (see \cite{ref29,ref9}) that almost surely, 
\begin{equation}
\label{covering1}
[0,1] =  \limsup_{j\rightarrow +\infty} \bigcup_{n\in A_j} B(X_n,B_n^{-   (1-\widetilde{\varepsilon}_j)})   .
\end{equation}
Hence, almost surely, every $x\in [0,1]$  is infinitely many times at distance less than $B_n^{-\eta(1-\widetilde{\varepsilon}_j)}$ from a point $X_n$. 

A more subtle covering theorem is needed, using only isolated points $(X_n)_{n\in \mathcal{I}_ j}$  (instead of $(X_n)_{n\in A_ j}$).
\begin{theorem}\label{prop5.1}
With probability one, $G'_1 = [0,1]$.
\end{theorem}

\begin{proof}
For $j\in\mathbb{N}$,   define the following set
$$
D_j = \{ [8k2^{-\lfloor\eta j\rfloor}, (8k+1) 2^{-\lfloor\eta j\rfloor}] \ : \ 0\leq 8k < 2^{\lfloor\eta j\rfloor} -1\} .
$$
Obviously, $\card  (D_j ) \sim 2^{\lfloor \eta j\rfloor}/8$.

For all $V\in D_j$ (necessarily, $V\subset [0,1]$),  consider the following event:
\begin{eqnarray}
\label{defa}
&& \ \ \ \ \ \mathcal{A}_j (V) \\
\nonumber&& =  \left\{\exists \, n\in A_j \ \mbox{ such that } X_n\in V \ \mbox{ and } \ B(X_n, 2B_n^{-1-\gamma })\bigcap \bigcup_{m\in \widetilde{A}_j}\{X_m\}=\{X_n\} \right\} 
\end{eqnarray}

\begin{lem}
\label{lemiso}
If $\mathcal{A}_j (V) $ is realized, then a point  $X_n$ given by \eqref{defa} is isolated in the sense of Definition \ref{def9}.
\end{lem}
\begin{proof}
When $\mathcal{A}_j (V)$ is realized, the point $X_n$ is such that for every $m\in\widetilde{A}^{j}$,   $X_m\notin B(X_n,2B_n^{-1-\gamma})$. 

Further, recall that  $ 2^{(j-1) \eta }\leq B_n\leq 2^{-j \eta}$, and that  $B_n^{-1/\eta} < B_n^{-1-\gamma} $  by our choice for  $\gamma$.
In addition, observe that when    $m\in\widetilde{A}_{j}$ for $j$ sufficiently large,
\begin{align*}
 B_m ^{-1/\eta} & \leq 2^{-(\lfloor (1-p\eta\varepsilon_j)j \rfloor-1)/\eta } \leq  j^{p\eta} 2^{- j/\eta+1/\eta    } \leq B_n ^{-1-\gamma }  ,
 \end{align*}  
again due to our choice for  $\gamma$. 
 
 What precedes proves that  $B(X_n,B_n^{-1/\eta}) \cap B(X_n,B_n^{-1/\eta}) =\emptyset$, hence $X_n$ is isolated. 
 \end{proof}
 
Our goal is now to prove that these events $\mathcal{A}_j (V)$  are realized very frequently.

The restrictions of  the point Poisson  process $\{(X_n,B_n \}_{n\in\mathbb{N}}$ on $V\times [1,+\infty]$, or equivalently of $\{(X_n,B_n^{-\frac{1}{\eta}})\}_{n\in\mathbb{N}}$ on $V\times [0,1]$, on the  dyadic intervals $V\in D_j$, are independent. Moreover, the intervals in $D_j$ being pairwise distant from at least $2^{1-\eta j}$, and since $B_n^{-1-\gamma }\leq B_n^{-1}\leq 2^{1-\eta j}$,  two balls $B(X_n,2B_n^{-1-\gamma  })$ with $X_n\in V$ and $B(X_m,B_m^{-1-\gamma})$  with $X_m \in  V' \neq V$ (with $n,m\in A_j$) do not intersect.  As a conclusion, the events $\mathcal{A}_j (V)$ for $V\in D_j$ are independent.

We introduce the set of (random) intervals
$$Q_j = \{ V\in D_j \ :  \ \mathcal{A}_j (V) \mbox{ is true }\}.$$
Let $V\in D_j$ with $V\subset[0,1]$, and  consider the random variable $T_j (V) =\mathds{1}_{\mathcal{A}_j (V)}$. From the above considerations,      the random variables $(T_j (V))_{V\in D_j}$ are i.i.d. random Bernoulli variables   with common parameter $p_j (1+\gamma )=\mathbb{P}(\mathcal{A}_j (V) \ \mbox{ is true})$. Since $\card(Q_j) = \sum_{V\in D_j} T_j (V)$,  $\sum _{V\in D_j} T_j (V) \sim \mathcal{B}(\card (D_j) ,p_j(1+\gamma))$, a binomial law with parameters $\card(D_j)$ and $p_j(1+\gamma)$.

The parameter is denoted   $p_j(1+\gamma)$ because, the law of the random variables  $X_n$ and $B_n$  being given, it depends only on $\gamma$ and $j$.
To go further, we call for the following lemma that is proved in  \cite{ref29}, Lemma 28 (see also  \cite{ref30}).

\begin{lem}
There exists a continuous function $k:(1,+\infty)\rightarrow ]0,1[$ such that for any $j\in\mathbb{N}^*$, $p_j(\delta) \geq k(\delta)>0$.
\end{lem}
 Let $(j_p)_{p\in\mathbb{N}^*}$ be  the increasing sequence of  integers defined iteratively by $j_1=1+\lfloor p\eta\varepsilon_1\rfloor$ and $j_{p+1}=\lfloor 2(1/\eta+1)j_p+1 \rfloor$. By construction,  $\widetilde{A}_{j_p}\cap\widetilde{A}_{j_{p+1}}=\varnothing$.

\medskip

Two intervals $V$, $V' \in D_j $ are called  {\em successive} when writing $V=[8k2^{-\lfloor\eta j\rfloor}, (8k+1) 2^{-\lfloor\eta j\rfloor}]$, then  either $V' = [8(k+1)2^{-\lfloor\eta j\rfloor}, (8(2+\gamma)+1) 2^{-\lfloor\eta j\rfloor}]$  or  $V' = [8(k-1)2^{-\lfloor\eta j\rfloor}, 8k 2^{-\lfloor\eta j\rfloor}]$.

Next lemma shows that it is highly likely that amongst any set of $j_p\log j_p$ successive intervals in $D_j$, at least one of them, say $V$, satisfies $\mathcal{A}_j (V) $.

\begin{lem}
\label{lemep}
For all $p\in\mathbb{N}$,   define the events $\mathcal{E}_p$ 
 by
\begin{eqnarray*}
\mathcal{E}_p&=&\{\mbox{for all } (V_1,...,V_{\lfloor j_p \log j_p \rfloor}) \mbox{ successive intervals of }  {D}_{j_p} \nonumber\\
& &\ \exists k\in\{1,...,\lfloor j_p\log j_p\rfloor\}  \ \mbox{ such that } \mathcal{A}_{j_p} (V_k) \ \mbox{ is true}\} , 
\end{eqnarray*}
Then $\mathbb{P}(\limsup\limits_{p \rightarrow +\infty} \mathcal{E}_p  ) = 1 $.   
\end{lem}

\begin{proof}
It is  easily  checked that the   $\{\mathcal{E}_p\}_{p\in\mathbb{N}}$ are mutually  independent by our choice for $(j_p)_{p\geq 1}$.
 There is a constant $K> 0$ such that
\begin{eqnarray*}
\mathbb{P}(\mathcal{E}^c_p) 
&\leq & \sum_{i=1}^{\card(D_{j_n})} \prod_{k=1}^{\lfloor j_n\log j_n \rfloor} \mathbb{P}(\mathcal{A}_{j_n} (V_k) \mbox{ is false}) \nonumber\\
& \leq & K2^{\eta j_p } (1-p_j(1+\gamma))^{j_p \log j_p} \nonumber\\
& \leq & K2^{\eta j_p} (1-k(1+\gamma))^{j_p \log j_p}. \nonumber
\end{eqnarray*}

By construction, $j_p>\!\!\!> p$ and $0< 1-k(1+\gamma)<1$. This implies  that for $p$ large enough, there exists $K'>0$ such that $\mathbb{P}(\mathcal{E}^c_p)\leq K'e^{- p}$, and so  $ \mathbb{P}(\mathcal{E}_p)\geq 1- K'e^{- p} $.

In particular, $\sum\limits_{p\in\mathbb{N}}\mathbb{P}(\mathcal{E}_ p )=+\infty$, and  Borel-Cantelli's lemma  yields  the result.
 \end{proof}

 Let $p$ be such that $\mathcal{E}_p$ is realized (this happens for an infinite number of $p$'s). 
  
Soit $V\in D_{j_p}$ such that $A_j  (V)$ holds true. Hence   $V$ contains an isolated point, by Lemma  \ref{lemiso}.

From the $\mathcal{E}_p$'s and Lemma \ref{lemep},  it follows that amongst any $\lfloor {j_p} \log {j_p} \rfloor$ consecutive intervals in $D_{j_p}$ there is at least one interval that contains an isolated point.  Consequently,  
 $$\bigcup\limits_{n\in \ii _{j_p}} B(X_n,8 j _p \log {j_p}2^{-\eta {j_p}})$$ 
 forms a covering of $[0,1]$. Since this occurs for an infinite number of integers ${j_p}$, and recalling \eqref{defg'delta} and the definition of $\widetilde \ep_j$, we   conclude that almost surely, 
 $$ [0,1] =   \limsup\limits_{j\rightarrow +\infty} \bigcup\limits_{n\in \ii _j} B(X_n,  8j\log {j}   2^{-\eta   j}) \subset \limsup_{j\rightarrow +\infty} \bigcup_{n\in \ii _j} B(X_n,B_n^{- (1-\widetilde{\varepsilon}_j)})= G'_1  ,$$
since $B_n \geq 2^{(j_p-1)/\eta}$ when $n \in \ii _{j_p}$. Hence the result.
\end{proof}

\section{Uniform regularity} \label{sec:4}

In this section, the uniform H\"older regularity of $F$  is investigated.

Recall that $\alpha \in ]0,1[$ and    $\psi$  is Lipschitz.
  
  An important tool for the following proofs is the wavelet transform. It is known since Jaffard's works that wavelets provide a convenient method to analyse pointwise regularity of functions.
   \begin{defn}
Let $\phi: \R\to\R$ be a compactly supported, non-zero function, with a vanishing integral: $\displaystyle \int_\R \phi(u)du=0$.

The continuous wavelet transform associated with $\phi$ of a function $f\in L^2(\mathbb{R})$ is  defined  for every couple   $(s,t)\in\mathbb{R}_{+}^*\times\mathbb{R}$  by
\begin{equation}\label{eq7.1}
W_f (s,t)=\frac{1}{\sqrt{s}} \int_{\mathbb{R}} f(x)\phi_{s,t} (x) dx \ \ \mbox{ where }\ \ \phi_{s,t}(x)=\phi\left(\frac{x-t}{s}\right).
\end{equation}
\end{defn}

Recall here the theorem of  Jaffard \cite{ref24} and Jaffard-Meyer \cite{ref26} relating the decay rate of continuous wavelets and uniform regularity for a function $f$.

\begin{theorem}\label{thm7.11}
Let $H\in\mathbb{R}_+^*$, $f\in L^\infty_{loc} (\mathbb{R})$, and $\psi$ be sufficiently regular (if $\alpha\in ]0,1[$ then $\psi$ is a Lipschitz function, otherwise $\psi \in C^{\lfloor\alpha\rfloor +1} (\RR)$). Then,  the mapping     $x\mapsto |x|^H |\log|x|| ^\beta $ is a uniform modulus of continuity for $f$   if and only if there exists a constant $K>0$ such that 
$$\forall (s,t)\in\mathbb{R}_+^*\times\mathbb{R}, \ |W_f (s,t)|\leq K s^{H+\frac{1}{2}} |\log |s| |^\beta.$$
\end{theorem}

Next proposition deals with the uniform regularity of $F$.

\begin{proposition}\label{prop4.2}
Almost surely, for $\alpha\in \mathbb{R}_+^*\backslash\mathbb{N}$, $\eta\in\mathbb{R}_+^*$, $\alpha\eta <1$ and $\psi$ sufficiently regular  as in Theorem \ref{thm7.11}.  Almost surely, there exists $K>0$ such that  for any $(s,t)\in [0,1]^*\times\mathbb{R}$
$$|W_F (s,t) | \leq K s^{\alpha\eta+\frac{1}{2}} |\log_2 (s)|^{2+\alpha}.$$
Therefore, item (i) of Theorem \ref{mainth2} holds true.
\end{proposition}

\begin{proof}
Let $(s,t)\in\mathbb{R}_{+}^*\times\mathbb{R}$. Note that the wavelet transform $W_F$ of $F$ can be expanded in
$$
W_F (s,t) =\frac{1}{\sqrt{s}} \int_{\mathbb{R}} F(x) \phi_{s,t}(x)dx = \sum_{n=1}^{+\infty} C_n^{-\alpha} d_n (s,t)
$$
with
\begin{equation}\label{eq7.5}
d_n (s,t)=\frac{1}{\sqrt{s}}\int_{\mathbb{R}}\psi_n (x)\phi_{s,t}(x)dx.
\end{equation}

A quick computation   allows to bound by above  $|d_n|$ (see Proposition  2.2.1 \cite{ref5}).
\begin{lem}\label{lem7.1} 
There exists $K>0$ such that
\begin{equation}\label{eq7.6}
\forall (s,t)\in  [0,1]\times\mathbb{R},\ |d_n(s,t)|\leq K s^{\frac{1}{2}} \min\{ s B_n^{\frac{1}{\eta}},s^{-1} B_n^{-\frac{1}{\eta}} \} T_n(t,s).
\end{equation} 
\end{lem}  

Fix $t\in \R$ and $0<s<1$.   there exists a unique  $J\in\mathbb{N}$ such that $2^{-\eta J+1}\leq s < 2^{-\eta J}$. 

When $j\leq \eta J$ and  $n\in A_j$, one has   $ \min\{ s B_n^{\frac{1}{\eta}},s^{-1} B_n^{-\frac{1}{\eta}} \}  = s B_n^{\frac{1}{\eta}} \leq s 2^j $. Also, by Lemma \ref{lem3.6}, $\sum_{n\in A_j} T_n (t,2^{-\eta J} ) \leq K j^2$.  So,  by  Lemma  \ref{lem7.1} and \eqref{eq3.9}, there exists a constant $K_1 >0$ (whose value can change from line to line, but does not depend on $s$, $t$, $j$ or $J$) such that 
\begin{eqnarray}
 {\sum_{j=0}^{\lfloor \eta J \rfloor}}\sum_{n\in A_j} C_n^{-\alpha} |d_n (s,t)| & \leq  & K_1 s^{\frac{1}{2}} \sum_{j=0}^{\eta J } 2^{-\alpha\eta(1-\varepsilon_j)j} s 2^j \sum_{n\in A_j} T_n (t,s) \nonumber \\
 & \leq  & K_1 s^{\frac{1}{2}} \sum_{j=0}^{\eta J} 2^{-\alpha\eta(1-\varepsilon_j)j} s 2^j \sum_{n\in A_j} T_n (t,2^{-\eta J} ) \nonumber \\
& \leq & K_1 s^{\frac{3}{2}}\sum_{j=0}^{\eta J} j^{2+\alpha} 2^{(1-\alpha\eta)j} \leq  K' s^{\frac{3}{2}} (\eta J)^{2+\alpha} 2^{(1-\alpha\eta) \eta J} \nonumber\\
& \leq & K_1 s^{\alpha\eta+\frac{1}{2}} |\log_2 (s)|^{2+\alpha}. \nonumber
\end{eqnarray}

When $ \eta J +1\leq j\leq J$, if $n\in A_j$ then  $ \min\{ s B_n^{\frac{1}{\eta}},s^{-1} B_n^{-\frac{1}{\eta}} \}  = s^{-1} B_n^{-\frac{1}{\eta}} \leq  s^{-1}2^{-j}$  and  Lemma \ref{lem3.6} still gives  $\sum_{n\in A_j} T_n (t,2^{-\eta J} ) \leq K j^2$.  Hence, there exists   $K_2 >0$  such that 
\begin{eqnarray}
 {\sum_{j=\lfloor \eta J \rfloor+1}^{ J }}\sum_{n\in A_j} C_n^{-\alpha} |d_n (s,t)| & \leq  & K_2 s^{\frac{1}{2}}\sum_{j=\eta J+1}^{ J } 2^{-\alpha\eta(1-\varepsilon_j)j} s^{-1 } 2^{-j} \sum_{n\in A_j} T_n (t, 2^{-\eta J} ) \nonumber \\
& \leq & K_2 s^{ -\frac{1}{2}} \sum_{j=\eta J+1}^{ J }  j^{2+\alpha} 2^{-(1+\alpha\eta)j} \leq  K_2 s^{-\frac{1}{2}} J^{2+\alpha} 2^{ -(1+\alpha\eta)\eta J} \nonumber\\
& \leq & K_2 s^{\alpha\eta+\frac{1}{2}} |\log_2 (s)|^{2+\alpha}. \nonumber
\end{eqnarray}

Finally, when $j\geq J$,  $ \min\{ s B_n^{\frac{1}{\eta}},s^{-1} B_n^{-\frac{1}{\eta}} \}    \leq  s^{-1}2^{-j}$ and    Lemma \ref{lem3.6} yields this time $\sum_{n\in A_j} T_n (t,2^{-\eta J} ) \leq K j^2 2^{\eta(j-J)}$. Hence, there exists $K_3 >0$ such  that 
\begin{eqnarray}
\sum_{j=J}^{+\infty}\sum_{n\in A_j} C_n^{-\alpha} |d_n (s,t)| & \leq  & K_3 s^{\frac{1}{2}} \sum_{j=J}^{+\infty } 2^{-\alpha\eta(1-\varepsilon_j)j} s^{-1} 2^{-j} \sum_{n\in A_j} T_n (t,2^{-\eta J}) \nonumber \\
& \leq & K_3 s^{-\frac{1}{2}}\sum_{j=J}^{+\infty } j^{2+\alpha} 2^{-(1+\alpha\eta)j}  2^{\eta(j-J)}  \leq K_3 s^{-\frac{1}{2}} J^{2+\alpha} 2^{-(1+\alpha\eta)J} \nonumber\\
& \leq & K_3 s^{\alpha\eta+\frac{1}{2}} |\log_2 (s)|^{2+\alpha} \nonumber.
\end{eqnarray}

The combination of the previous inequalities yields that  for some constant $K>0$,
\[|W_F (s,t)|\leq K s^{\alpha\eta+\frac{1}{2}} |\log_2 (s)|^{2+\alpha}.\]
Theorem \ref{thm7.11} allows to conclude the proof of Proposition \ref{prop4.2}.
\end{proof}

\section{Lower-bound for the H\"older exponent of $F$ via  the study of $G_\delta$} \label{sec:6}

When $\delta\in [1,\frac{1}{\eta}]$, next proposition  yields a lower bound for the pointwise H\"older exponent of $F$ at $x_0$ when $x_0 \notin G_\delta$. 

\begin{proposition}\label{prop6.1}
 Almost surely, for every   $\delta\in (1,\frac{1}{\eta})$, for every  $x_0\notin G_\delta$, there exists $K_{x_0} >0$ such that for any $x$ close to $x_0$,
$$|F(x)-F(x_0)| \leq K_{x_0}  |\log_2 |x-x_0|\ |^{2+\alpha} |x-x_0|^{\frac{\alpha}{\delta}} .$$
Therefore, $h_F(x_0) \geq \frac{\alpha}{\delta}$.
\end{proposition}

\begin{proof} Let $x_0\notin G_\delta$. For $x$ with $|x-x_0|\leq 1$,  there exists a unique $j_0\in\mathbb{N}$ such that 
\begin{equation}
    2^{- \eta (j_0+1)} \leq |x-x_0| <2^{- \eta j_0} \nonumber
\end{equation}
and call $j_1 $ the largest positive  integer so that $ |x-x_0|+2^{-j_1} \leq 2^{-\delta\eta j_1}$. The integer $j_1$ exists since $2^{-j\eta \delta j_1}$ tends to 0 when $j_1 \to +\infty$.

Observe that when $j_0$ becomes large,   $|j_1 - j_0 /\delta|  \to 0$. So it is assumed that $j_0$ is so large that $ j_0/ \delta \leq j_1  \leq   j_0/ \delta +2$, so that $2^{-j_0\eta}\sim  2^{-j_1\delta \eta} \sim |x-x_0|$. Observe also that this explains the fact that $\delta$ must be less or equal than $1/\eta$.

By definition of $G_\delta$,   since $x_0\notin G_\delta$, there exists at most a finite number, say  $N_{x_0} $, of  balls  $ \{B (X_{n_k},B_{n_k}^{-\delta} )  \}_{1\leq k \leq N_{x_0} }$ 
that contain $x_0$. Write $\widetilde j_0$ for the smallest integer $j$ such that $\bigcup_{k=1}^{N_{x_0} } \{n_k\}  \in \bigcup_{j=1}^{\widetilde j_0} A_j$.  So it may be assumed that $x$ is so close to $x_0$ that  for every $j\geq j_1/2\delta $, $j\ep_j \geq   \widetilde j_0 +1$ and for every $n \in A_j $ with $j\geq   j_1$, $|x_0-X_n| > B_{n}^{-\delta}$.

Recalling that  the support of $\psi_n$ is the ball $B(X_n, B_n^{-1/\eta})$ and that $\delta \leq 1/\eta$, this implies that     $x_0$   belongs to the  support  of  at most $N$ pulses $\psi_n$ with $n\in A_j$ and $j < j_1$, and does not belong to any support of $\psi_n$, for $n\in  A_j$ and $j\geq j_1$.

Also, when $j\leq j_1$ and $n\in A_j$, by definition of $j_1$, one has $|x-x_0|+B_n^{-1/\eta} \leq B_n^{-\delta}$. Hence  $x \in B(X_n, B_n^{-1/\eta})$  would imply that    $x_0 \in   B(X_n, B_n^{-\delta})$, which is possible for only $N$ balls.  Consequently,   $x$ and $x_0$ both belong to at most $N$ supports of pulses $\psi_n$  with $n\in A_j$ and $j\leq j_1$. 

Let us write  $ 
|F(x)-F(x_0)|\leq S_1+S_2+S_3   
$
with
$ F_j(x) = \sum_{n\in A_j} C_n^{-\alpha} \psi_n(x)$ and 
\begin{eqnarray}
S_1  = \left | \sum_{j=0}^{j_1-1} F_j(x)-F_j(x_0) \right |,  \ \ \ \ 
S_2  =  \sum_{j=j_1}^{+\infty} |F_j(x_0)| \  \  \mbox{ and } \ \ 
S_3 =  \sum_{j=j_1}^{+\infty} |F_j(x)| . \nonumber 
\end{eqnarray}
 We first  give an upper-bound for   $ S_1 $.  By the remarks above, $S_1$ contains at most $N_{x_0} $ non-zero terms of the form $C_{n_i} ^{-\alpha} (\psi_{n_i}(x)-\psi_{n_i}(x_0))$ (for integers $n_1$, ..., $n_{N_{x_0} }$),   and for each of them, since     $\psi $ is Lipschitz with some constant $K>0$, one has 
$$
  C_{n_i} ^{-\alpha} \left|\psi\left( B_{n_i} ^{\frac{1}{\eta}} \left(x-X_{n_i} \right) \right)-\psi\left( B_{n_i} ^{\frac{1}{\eta}} \left(x_0-X_{n_i} \right) \right) \right|  
    \leq  C_{n_i} ^{-\alpha} B_{n_i} ^{\frac{1}{\eta}}K |x-x_0|    . $$
    By (\ref{eq3.4}), (\ref{eq3.9}) and the definition of $\widetilde j_0$,  if ${n_i}  \in A_j$, then one has for some other constant $K>0$ that    
    $$
   C_{n_i} ^{-\alpha} B_{n_i} ^{\frac{1}{\eta}}    \leq  K   2^{-\alpha\eta j(1-\varepsilon_{j})}2^j  \leq K  \widetilde j_0 ^\alpha 2^{\widetilde j_0(1- \alpha\eta)} \leq K   j_1  ^\alpha 2^{ \ep_{j_1}  j_1  }   =  K j_1^{\alpha+1/\eta}.$$
   
Using that  $j_1\sim \delta j_0\sim \frac{\delta}{\eta} |\log_2|x-x_0|| $, this finally gives   for some constant $K_{x_0}$  depending on $x_0$ \begin{eqnarray}
\nonumber
S_1 & \leq & K N_{x_0}   |x-x_0|   j_1 ^{\alpha+1/\eta}   \leq K_{x_0}   |x-x_0| \cdot  | \log_2 |x-x_0|\ |^{ \alpha+1/\eta}\\
\label{eq_S1}
& \leq&    |x-x_0| ^\alpha  | \log_2 |x-x_0|\ |^{ 3+ \alpha } .
\end{eqnarray} 
 Observe that the last inequality holds when $j_1$ tends to $+\infty$, and is quite crude.
 
 \medskip

By construction, $\psi_n(x_0) = 0$  for    every  $n\in A_j $ with $j\geq j_1$, so $S_2=0$.

\medskip

Finally, for   $S_3$, one   writes that $ |\psi_n(x)    | \leq ||\psi||_{\infty}$, and then \begin{eqnarray}
S_3 & = & \sum_{j=j_1}^{+\infty} |F_j(x)| \leq      K ||\psi||_{\infty} \sum_{j=j_1}^{+\infty}    \sum_{n\in A_j} C_{n} ^{-\alpha} {\bf 1 \!\!\!1}_{\psi_n(x) \neq 0} \\
&  \leq &      K ||\psi||_{\infty} \sum_{j=j_1}^{+\infty} j^{\alpha}  2^{-\alpha \eta j} \sum_{n\in A_j}  T_n(x,0)\nonumber\\
\label{eq6.4}
 & \leq &  K ||\psi||_{\infty} \Big( \sum_{j=j_1}^{+\infty} j^{\alpha}  2^{-\alpha \eta j} j^2   \Big)  \leq K j_1^{2+\alpha} 2^{-\alpha \eta j_1 }
\leq  K  j_0^{2+\alpha} 2^{- j_0\frac{\alpha\eta}{\delta}}  \nonumber \\
 &\leq &  K | \log_2 |x-x_0|\ |^{2+\alpha} \ |x-y|^{\frac{\alpha}{\delta}} .
\end{eqnarray} 
The result follows from    (\ref{eq_S1}) and (\ref{eq6.4}), and by letting     $\ep$ tend to zero.
\end{proof}

\section{Upper-bound for the H\"older exponent of $F$ via the sets $G'_\delta$} \label{sec:7}

We  now     find an upper bound for  the pointwise  H\"older exponent of $F$ at every $x_0\in G'_\delta $, using a wavelet method. Let us recall  the theorem of Jaffard \cite{ref24} relating continuous wavelet transforms and pointwise regularity.

\begin{theorem}\label{thm7.1}
Let $f\in L^\infty_{loc}(\mathbb{R})$, $x_0\in \R$ and $H>0$. If $f\in C^H (x_0)$, then there exists $K>0$ and a neighborhood $U$ of $(0^+,x_0)$ such that
$$
\forall  (s,t)\in U \ \ ,\ \ |W_f (s,t)|\leq K s^{ \frac{1}{2}} (s + |{x_0-t}| )^H.
$$
\end{theorem}

This theorem is key to prove next proposition.

\begin{proposition}\label{prop7.1}
Almost surely, for all $\delta\in \left[1,\frac{1}{\eta}\right]$ and  $x_0\in G'_\delta$,   $h_F (x_0) \leq \frac{\alpha}{\delta}$.
\end{proposition}

\begin{proof}
First, without loss of generality,   assume in addition that the function $\phi$ used to compute the wavelet transform belongs to  $C^1(\R)$, is  exactly supported by the interval $[-1,1]$, and that 
\begin{equation}
\label{property_phi}
\int_{-1}^1 \phi(u) \psi(u) du \neq 0.
\end{equation}
The existence of such a $\phi$ is a trivial exercise.

\smallskip

Fix $x_0\in G'_\delta$. There exist two increasing  sequences of integers $(n_k)_{k\in\mathbb{N}}$ and   $(j_k)_{k\in\mathbb{N}}$  such that  $n_k\in  \ii_{j_k}$  and $x_0\in  B(X_{n_k},B_{n_k}^{-\delta(1-\widetilde{\varepsilon}_{j_k})}).$
 
Let $k\in\mathbb{N}^*$ with $n_k\in \ii_{j_k}$.  The values of continuous wavelet transforms $W_F (B_{n_k}^{-\frac{1}{\eta}}, X_{n_k})$, are now estimated. Setting  $J_k=\lfloor(1-p_0\eta\varepsilon_{j_k})j_k \rfloor$ and $\widetilde{J}_k=\lfloor(1+\gamma)j_k \rfloor$, one writes $W_F(B_{n_k}^{-\frac{1}{\eta}},X_{n_k}) = S_1 + S_2 + S_3$
with
\begin{eqnarray}
S_1 &=& \sum_{j=0}^{J_k-1}\sum_{n\in A_j}C_n^{-\alpha} d_n(B_{n_k}^{-\frac{1}{\eta}},X_{n_k}), \ \ \ \
S_2 = \sum_{j=J_k}^{\widetilde{J}_k}\sum_{n\in A_j}C_n^{-\alpha} d_n(B_{n_k}^{-\frac{1}{\eta}},X_{n_k}) \nonumber \\
\mbox{ and } \ \
S_3 &=& \sum_{j=\widetilde{J}_k+1}^{+\infty}\sum_{n\in A_j}C_n^{-\alpha} d_n(B_{n_k}^{-\frac{1}{\eta}},X_{n_k}). \nonumber
\end{eqnarray}
 
Let us first find a lower bound for   $S_2$. Recalling  the definition  (\ref{eq5.2}) of $\ii_{j_k}$,  $ {n}_k $ is the unique integer in $\widetilde  A_{j_k}$ such that $x_0\in B(X_{\widetilde{n}_k},B_{\widetilde{n}_k}^{-\frac{1}{\eta}})$. Hence, recalling \eqref{eq7.5}, $d_n (B_{n_k}^{-\frac{1}{\eta}}, X_{n_k}) =0$ when $n\neq n_k$ (since the support of $\psi_n$ and $\phi_{n_k} $ do not intersect)  and 
\begin{equation}
    S_2 = C_{n_k}^{-\alpha} d_{n_k} (B_{n_k}^{-\frac{1}{\eta}}, X_{n_k}). \nonumber
\end{equation}
An integration by part and a change of variables give
$$
d_{n_k} (B_{n_k}^{-\frac{1}{\eta}},X_{n_k})  = B_{n_k}^{-1/(2\eta)} \int_{-1}^{1} \psi(u)\phi(u)du.$$ 
Condition \eqref{property_phi} implies  that for some fixed constant $K_2>0$ (depending on $\psi$ and $\phi$ only), for every integer $k$,  
\begin{equation}
\label{minS2}
    |S_2|\geq K_2  C_{ {n}_k}^{-\alpha} B_{n_k}^{-\frac{1}{2\eta}} \geq  K _2B_{n_k}^{-\frac{1}{2\eta}}  2^{-\alpha\eta (1+\varepsilon_{j_k})j_k} \geq  K_2  B_{n_k}^{-\frac{1}{2\eta}-\alpha (1+\varepsilon_{j_k})}.
\end{equation}
where   (\ref{eq5.2}) and (\ref{eq3.9}) have been used.

Next, let us estimate $S_1$. By  (\ref{eq7.6}),  (\ref{eq3.9}) and (\ref{eq3.4}), one has
\begin{eqnarray}
|S_1| & \leq &  \sum_{j=0}^{J_k-1} \sum_{n\in A_j} C_n^{-\alpha} |d_n(B_{n_k}^{-\frac{1}{\eta}},X_{n_k})| \nonumber\\
& \leq & \sum_{j=0}^{J_k-1} \sum_{n\in A_j} C_n^{-\alpha} B_{n_k}^{-\frac{1}{2\eta}} \min\{ B_{n_k}^{-\frac{1}{\eta}} B_n^{\frac{1}{\eta}},B_{n_k}^{\frac{1}{\eta}} B_n^{-\frac{1}{\eta}} \} T_n(X_{n_k},B_{n_k}^{-\frac{1}{\eta}}) \nonumber\\
& \leq & \sum_{j=0}^{J_k-1}  2^{-\alpha\eta j(1-\varepsilon_j)} B_{n_k}^{-\frac{1}{2\eta}} \min\{ B_{n_k}^{-\frac{1}{\eta}} 2^j,B_{n_k}^{\frac{1}{\eta}} 2^{-j-1} \} \sum_{n\in A_j} T_n( X_{n_k}, 2^{-j_{k}}) . \nonumber 
\end{eqnarray}
When $j<(1-\eta\varepsilon_{j_k})j_k$, $ B_{n_k}^{-\frac{1}{\eta}}  \leq 2^{-j-1}$, so the minimum above is less than  $2 B_{n_k}^{-\frac{1}{\eta}} 2^j$. In addition, by \eqref{eq3.5} one has 
$
 \sum_{n\in A_j} T_n( X_{n_k}, 2^{-j_{k}})  \leq  K  j^2  
$ (this holds as long as $j \leq j_k/\eta$).
Hence by  (\ref{eq3.11}), for some constant $K_1 >0$ (that may change from one inequality to the next one), 
\begin{eqnarray*}
|S_1|& \leq & K_1 \sum_{j=0}^{J_k-1} j^{2+\alpha} 2^{-\alpha\eta j} B_{n_k}^{-\frac{1}{2\eta}} B_{n_k}^{-\frac{1}{\eta}} 2^j \leq K_1  B_{n_k}^{ -\frac{3}{2\eta}} \sum_{j=0}^{J_k-1} j^{2+\alpha} 2^{(1-\alpha\eta) j} \\
&\leq & K_1 B_{n_k}^{-\frac{3}{2\eta}} j_k^{2+\alpha} 2^{(1-\alpha\eta)(1-p_0\eta\varepsilon_{j_k})j_k}.
\end{eqnarray*}
Since $j_k=2^{\eta\varepsilon_{j_k} j_k}$ and $n_k\in \ii_{j_k}$, $2^{j_k}\leq  B_{n_k}^{\frac{1}{\eta}}$,  so 
\begin{eqnarray}
    |S_1| &\leq & K_1  B_{n_k}^{-\frac{3}{2\eta}} B_{n_k}^{(3+\alpha)\varepsilon_{j_k}} B_{n_k}^{(\frac{1}{\eta}-\alpha)(1-p_0\eta\varepsilon_{j_k})} \leq  K_1 B_{n_k}^{-\frac{1}{2\eta}-\alpha-(p_0-3-\alpha-\alpha\eta p_0)\varepsilon_{j_k}}. \nonumber
\end{eqnarray}
Our choice \eqref{choice-p} for $p_0$ ensures that  $p_0-3-\alpha-\alpha\eta p_0 > 2 \alpha$, hence 
\begin{eqnarray}
\label{majS1}
    |S_1| &\leq & K_1    B_{n_k}^{-\frac{1}{2\eta}-\alpha(1+2\ep_{j_k})}  .
\end{eqnarray}

Finally, for $S_3$, one writes  by  (\ref{eq7.6}), (\ref{eq3.9}) and (\ref{eq3.4}), and the same lines of computations as above,  that for some $K_3>0$, 
\begin{eqnarray}
|S_3| & \leq &  \sum_{j=\widetilde{J}_k+1}^{+\infty} \sum_{n\in A_j} C_n^{-\alpha} |d_n(B_{n_k}^{-\frac{1}{\eta}},X_{n_k})| \nonumber\\ 
& \leq &K_3  \sum_{j=\widetilde{J}_k+1}^{+\infty}  2^{-\alpha\eta j(1-\varepsilon_j)} B_{n_k}^{-\frac{1}{2\eta}} \min\{ B_{n_k}^{-\frac{1}{\eta}} 2^j,B_{n_k}^{\frac{1}{\eta}} 2^{-(j+1)} \} \sum_{n\in A_j} T_n(X_{n_k}, 2^{-j_k} ). \nonumber 
\end{eqnarray}
When $j \geq  \widetilde J_k = \lfloor  (1+\gamma)j_k \rfloor $, the above minimum is now reached at  $ B_{n_k}^{\frac{1}{\eta}}    2^{-j-1}$.

Then,  still by and (\ref{eq3.5}), the sum  $\sum_{n\in A_j} T_n(X_{n_k},  2^{-j_k} )$ is bounded above by $Kj^2$ when $j\leq j_k/\eta$, and by $Kj^2 2^{\eta(j-j_k/\eta)}$ when $j>j_k/\eta$. Hence by  (\ref{eq3.11}),   for  some constant $K_3$ that may change from line to line but does not depend on $k$ or any of the moving parameters,
\begin{eqnarray}
|S_3|& \leq &K_3 \sum_{j=(1+\gamma)j_k}^{\lfloor j_k/\eta \rfloor } j^{2+\alpha} 2^{-\alpha\eta j} B_{n_k}^{-\frac{1}{2\eta}} B_{n_k}^{\frac{1}{\eta}} 2^{-j}   \\ && \ \ \ + K_3 \sum_{j={\lfloor j_k/\eta \rfloor }+1}^{+\infty} j^{2+\alpha} 2^{-\alpha\eta j} B_{n_k}^{-\frac{1}{2\eta}} B_{n_k}^{\frac{1}{\eta}} 2^{-j}  2^{\eta(j-j_k/\eta )}\nonumber\\
\nonumber  &\leq &K_3   B_{n_k}^{ \frac{1}{2\eta} }\left(   \sum_{j=(1+\gamma)j_k}^{\lfloor j_k/\eta \rfloor } j^{2+\alpha} 2^{- (1+\alpha\eta) j}  +   2^{ - j_k}   \sum_{j=\lfloor j_k/\eta \rfloor +1}^{+\infty}j^{2+\alpha} 2^{(\eta-1-\alpha\eta )j} \right) .
\end{eqnarray}
The first sum above is bounded above by 
$$\sum_{j= \lfloor (1+\gamma)j_k\rfloor }^{\lfloor j_k/\eta\rfloor} j^{2+\alpha} 2^{-(1+\alpha\eta)j}  \leq K_3  {j}_k^{2+\alpha} 2^{-(1+\alpha\eta)(1+\gamma)j_k }$$
and the second one by
\begin{align*}
2^{-j_k} \sum_{j=\lfloor j_k/\eta \rfloor+1}^{+\infty} j^{2+\alpha} 2^{(\eta-1-\alpha\eta)j}  & \leq K_3 2^{-j_k}j_k^{2+\alpha} 2^{(\eta-1-\alpha\eta)j_k/\eta} = K_3 j_k^{2+\alpha}2^{-\frac{j_k}{\eta}(1+\alpha\eta)}.
\end{align*}
Since $B_{n_k}^{\frac{1}{\eta}} \sim 2^{j_k}$ and $j_k=2^{j_k \eta \ep_{j_k}} \sim  B_{n_k}^{\ep_{j_k}}$ and  $1+\gamma < 1/\eta $, we get that
\begin{eqnarray*}
|S_3| & \leq & K_3  {j}_k^{2+\alpha} 2^{-(1+\alpha\eta)(1+\gamma)j_k }  +  K_3 j_k^{2+\alpha}2^{-\frac{j_k}{\eta}(1+\alpha\eta)}\\
&\leq  &   K_3 B_{n_k} ^{-  \frac{(1+\alpha\eta)(1+\gamma)}{\eta} +(2+\alpha)\ep_{j_k}}.
\end{eqnarray*}
Observe that $\frac{(1+\alpha\eta)(1+\gamma)}{\eta} -(2+\alpha)\ep_{j_k} >  \frac{1}{2\eta}+\alpha (1+2\varepsilon_{j_k})$. So, 
\begin{equation}
\label{majS3}
|S_3| \leq K_3 B_{n_k}^{-\frac{1}{2\eta}-\alpha (1+2\varepsilon_{j_k})},
\end{equation}
this last inequality being very generous ($S_3$ is much smaller than the term on the right hand-side).

Combining  \eqref{minS2}, \eqref{majS1} and \eqref{majS3}, and the fact that $B_{n_k}^{ -\varepsilon_{j_k}} \to 0$ when $k$ tends to infinity, one concludes that  for every sufficiently large integers $k$, 
\begin{equation}
\label{eqfinal1}
|W_F(B_{n_k}^{-\frac{1}{\eta}},X_{n_k}) |\geq K B_{n_k}^{-\frac{1}{2\eta}-\alpha (1+\varepsilon_{j_k})}.
\end{equation}

 Assuming that $f\in \mathcal{C}^{\frac{\alpha}{\delta} +\ep}(x_0)$, we would have  by Theorem \ref{thm7.1} that for some $K'>0$, 
\begin{eqnarray*}
  |W_F(B_{n_k}^{-\frac{1}{\eta}},X_{n_k})   | &  \leq & K' B_{n_k}^{-\frac{1}{2\eta}}   \Big( B_{n_k}^{-\frac{1}{\eta}}  + |x_0-X_{n_k}| \Big)^{\frac{\alpha}{\delta} +\ep}   \\
&\leq &  K' B_{n_k}^{-\frac{1}{2\eta}}   \Big( B_{n_k}^{-\frac{1}{\eta}}  +B_{n_k}^{- \delta(1-\widetilde{\varepsilon}_{j_k}) }\Big)^{\frac{\alpha}{\delta} +\ep}  \\
&\leq & K' B_{n_k}^{-\frac{1}{2\eta}}     B_{n_k}^{- \delta(1-\widetilde{\varepsilon}_{j_k})(\frac{\alpha}{\delta} +\ep)}  \end{eqnarray*}
since $|x_0 -X_{n_k}| \leq B_{n_k}^{-\delta(1-\widetilde{\varepsilon}_{j_k})}$. This contradicts \eqref{eqfinal1} since  the sequences  $(\varepsilon_j)$ and $(\widetilde \varepsilon_j)$  converge to 0 as $j\to +\infty$. Consequently, $f\notin \mathcal{C}^{\frac{\alpha}{\delta} +\ep}(x_0)$ for every $\ep>0$,  hence the result.
\end{proof}

To conclude this part, we would like to emphasize that this analysis is quite sharp since the bounds obtained for $S_1$, $S_2$ and $S_3$ are very tight (and the choice for $p_0$ is key). Only the fine study of isolated points made it possible to obtain this result.

Also, observe that the proof does not work any more when  $\delta >1/\eta$, since in the last series of inequalities $|W_F (B_{n_k}^{-\frac{1}{\eta}},X_{n_k})| $, the term   $B_{n_k}^{-\frac{1}{\eta}} + B_{n_k}^{-\delta (1-\widetilde{\varepsilon}_{j_k})}$ can not be bounded by above by  $B_{n_k}^{-\delta (1-\widetilde{\varepsilon}_{j_k})}$.

\section{Multifractal spectrum of $F$} \label{sec:8}

Recall that  the study of  the regularity of $F$ is restricted to the interval $[0,1]$. We start by the range of possible exponents for $F$.

\begin{lem}
\label{lem_range}
Almost surely, for every $x\in \zu$, $\alpha\eta \leq h_F(x)\leq \alpha$.
\end{lem}
\begin{proof}
First,   Proposition \ref{prop4.2} yields that almost surely, for every $x\in \zu$, $h_F(x)\geq \alpha\eta$.

Then,  Theorem \ref{prop5.1}  gives  $ \zu =G'_1$, and   Proposition \ref{prop7.1} ensures that every $x\in G'_1$ satisfies $h_F(x)\leq \alpha$.  \end{proof}

Gathering the results proved in the previous sections (Propositions \ref{prop6.1} and \ref{prop7.1}, and Remark \ref{reminclusions}), one also sees that almost surely:
\begin{itemize}
\smallskip
\item
 for all $H\in [\alpha\eta,\alpha)$,  
  \begin{equation}
 \label{minEH}
  G'_{\alpha/H} \setminus \bigcup_{\delta >  \frac{\alpha}{H}} G_\delta \subset E_F(H).
    \end{equation}
 
 Indeed, when $x\in  G'_{\frac{\alpha}{H}} $, $h_F(x) \leq \frac{\alpha}{\frac{\alpha}{H}} = H$ and when $\delta >\frac{\alpha}{H} $ and  $x\notin  G_{\delta } $, $h_F(x) \geq  \frac{\alpha}{\delta}  $.
 
\smallskip
 \item
 for all $H\in [\alpha\eta,\alpha]$,  
 \begin{equation}
 \label{majEH}
  E_F(H)  \subset \bigcap _{\delta <  {\frac{\alpha}{H}}} G_\delta .
  \end{equation}

\end{itemize}

In order to obtain the multifractal spectrum of $F$,  a preliminary step consists in estimating   the Hausdorff dimension  and measures of the sets  $G_\delta$ and $G'_\delta$.

For  $h>0$, $\mathcal{H}^{h}$, $\mathcal{H}^{h}_\xi$ stand respectively   for the $h$-Hausdorff measure in $\R$ and the $\alpha$-Hausdorff
pre-measure computed with coverings of sets of diameter less than $\xi>0$.

\begin{proposition}
\label{propfinal2}
With probability one, for every $\delta \in [1,1/\eta]$, one has $\dim_H G_\delta \leq 1/\delta $  and $\mathcal{H} ^{1/\delta} (G'_\delta) =+\infty $.
\end{proposition}

\begin{proof}
The upper bound  $\dim_H G_\delta \leq 1/\delta $   follows by using   as coverings of $G_\delta$ the family $\{B(X_n, B^{-\delta}_n)\}_{j\geq J, n \in  A_j}$, for $J\geq 1$. For $\ep>0$,  
\begin{eqnarray*}
\mathcal{H}^{1/\delta+\ep}_{ 2^{-\eta J} } (G_\delta ) & \leq & \sum_{j\geq J} \sum_{n\in A_j} |B_n^{-\delta } |^{ 1/\delta+\ep}.
\end{eqnarray*}
By \eqref{eq3.5}, and using that $B_n\leq 2^{j\eta} $ when $n\in A_j$, one gets 
$$\mathcal{H}^{1/\delta+\ep}_{ 2^{-\eta J} }  (G_\delta ) \leq \sum_{j\geq J} 2^{\eta j(1+\varepsilon_j)}2^{-j\eta(1+\ep/\delta)},$$
 which  is the rest of a convergent series. Hence $\mathcal{H}^{1/\delta+\ep}(G_\delta ) =0 $ and $\dim_H G_\delta \leq 1/\delta+\ep$. 
 
  \medskip
  
  The fact that  $\mathcal{H} ^{1/\delta} (G'_\delta) =+\infty $ (giving the lower bound $\dim_H G'_\delta \geq 1/\delta$) is more delicate. The following  mass transference principle  \cite{BV,ref25} is useful.
  
\begin{theorem}\label{thm8.1}
Let $(x_n)_{n\in\mathbb{N}^*}$ be a real sequence in $[0,1]^d$  ($d\geq 1$) and $(\lambda_n)_{n\in\mathbb{N}^*}$ a decreasing sequence of   positive real numbers. For all $\delta \geq1$,  set
\begin{equation}
    L_\delta=\limsup_{n\rightarrow + \infty} B(x_n,\lambda_n^\delta)=\bigcap_{N\geq 1} \bigcup_{n\geq N} B(x_n,\lambda_n^\delta) \nonumber
\end{equation}
If the $d$-dimensional Lebesgue measure $\mathcal{L}(L_1)$ of $L_1$ equals 1,  then for all $\delta >1$, $\mathcal{H}^\frac{d}{\delta} (L_\delta)=+\infty$ and $\dimh  (L_\delta)\geq \frac{d}{\delta}$.
\end{theorem}
Theorem \ref{prop5.1} gives that $G'_1 =[0,1]$, almost surely.  In particular, $\mathcal{L}(G'_1)=1$. Applying the previous theorem to the (random) sequences $x_n=X_n$ and $\lambda_n = B_n^{-(1-\widetilde{\ep}_j)}$ when $n\in A_j$ yields  the claim of Proposition \ref{propfinal2}.
\end{proof}

We are now in position to conclude   the proof of Theorem \ref{maintheorem}.

\begin{proof}
First, by Lemma \ref{lem_range}, only  $H\in [\alpha\eta,\alpha]$ need to be considered.

Then, \eqref{majEH} yields that almost surely, $d^{[0,1]}_F(H) \leq \dimh G_\delta$, for every $\delta >\alpha/H$. Proposition \ref{propfinal2} yields $\dimh G_\delta \leq 1/\delta$, hence $d^{[0,1]}_F(H) \leq H/\alpha$.

Finally,  Proposition \ref{propfinal2} gives simultaneously that $\mathcal{H}^ {H/\alpha} (G'_{\alpha/H}) =+\infty $ and $\mathcal{H}^ {H/\alpha} (G_\delta )= 0$ for every $\delta <\alpha/H$.   So, $\mathcal{H}^ {H/\alpha} ( G'_{\alpha/H}   \setminus \bigcup_{\delta >  \frac{\alpha}{H}} G_\delta ) = +\infty$, and by \eqref{minEH}, $\mathcal{H}^ {H/\alpha}   (E_F(H) )=+\infty$. This gives $\dimh  E_F(H)  \geq H/\alpha$, and by the remarks above $d^{[0,1]}_F(H) = H/\alpha$.

When $H=\alpha$, the same argument gives that $\mathcal{L} (E_{F}(\alpha)\cap [0,1])=1$, i.e. $E_{F}(\alpha)$ is of full Lebesgue measure in $[0,1]$.
\end{proof} 

\section{Almost-everywhere   modulus of continuity} \label{sec:9}

 Le us explain how to obtain from what precedes the almost-everywhere   modulus of continuity for $F$, almost surely.
 
 \medskip
 
 By a Theorem by Jaffard-Meyer (Proposition 1.2 in \cite{ref26}),   the following (almost) equivalence holds true.
\begin{theorem}\label{thm9.1}
Let $f\in L^\infty_{loc}(\mathbb{R})$, $x_0\in \R$ and $H>0$.

If the  function $f$ has a local   continuity of continuity$\theta$ at $x_0$, then for some constant $C>0$ 
\begin{equation}
\label{eqmodule}
\forall  (s,t)\in U \ \ ,\ \ |W_f (s,t)|\leq K s^{ \frac{1}{2}}  ( \theta(s) +  \theta(|{x_0-t}| )).
\end{equation}
Conversely, if $f\in C^\ep(\R)$ for some $\ep>0$, and if \eqref{eqmodule} holds, then there exist   constants $\eta,C>0$  and a polynomial $P$ such that setting $j_0= \lfloor |\log_2|x-x_0|\rfloor$, one has
\begin{equation}
\label{eqmodule2}  \forall  x \mbox{ such that } |x-x_0|\leq \eta, \ \ | f(x) -P(x-x_0) |\leq  C \inf_{j\geq j_0} ( (j-j_0) \theta(|x-x_0|) +2^{-j\ep}) .
\end{equation}
 
\end{theorem}

 Observe that if $\theta(h)=|h|^\beta |\log |h||^\gamma$ with $ \ep< \beta<1$, then the infimum at the right hand side of \eqref{eqmodule2} is (roughly) reached at $j=j_0\beta/\ep $, and  \eqref{eqmodule2} reduces to
 $$ | f(x) -P(x-x_0) |\leq   C  |x-x_0| ^\beta |\log  |x-x_0|| ^{1+\gamma}.$$
 
 \medskip

 Coming back to Proposition \ref{prop7.1}, let $x_0\in G'_1$. At the end of the proof, recall the lower  bound \eqref{eqfinal1} for the wavelet coefficient $|W_F(B_{n_k}^{-\frac{1}{\eta}},X_{n_k}) |  \geq K B_{n_k}^{-\frac{1}{2\eta}-\alpha (1+\varepsilon_{j_k})}$.

 Remembering that   $ B_{n_k}  \sim 2^{  \eta j_k }   $, the formulas for $\ep _{j_k}$ and the fact that $\widetilde{\ep}_{j_k}$,   and 
  $|x_0-X_{n_k}|\leq  B_{n_k}^{- (1-\widetilde{\varepsilon}_{j_k})}$, one   successively has (for large integers $k$)
  \begin{align*}
 B_{n_k}^{ - \ep_{j_k}}  & \sim |\log j_k|  ^{-1} \geq  C | \log|x_0-X_{n_k}| |,\\
 B_{n_k}^{ - \widetilde \ep_{j_k}}  & \sim |\log j_k| ^{-1} \geq C | \log|x_0-X_{n_k}||,\\
 B_{n_k}^ {-1} &  \geq  C |x_0-X_{n_k}|  |\log  |x_0-X_{n_k}|  |,
\end{align*}
for some constant $C>0$  that depends on $\eta $ only. Hence,  
\begin{align*}
|W_F(B_{n_k}^{-\frac{1}{\eta}},X_{n_k}) | & \geq     K B_{n_k}^{-\frac{1}{2\eta}-\alpha (1+\varepsilon_{j_k})} 
\geq  K B_{n_k}^{-\frac{1}{2\eta}} B_{n_k}^{-\alpha }|  \log|x_0-X_{n_k}| |^\alpha\\ 
&  \geq  K C B_{n_k}^{-\frac{1}{2\eta}}    |x_0-X_{n_k}|^\alpha  |\log  |x_0-X_{n_k}|  | ^{2\alpha}\\
&  \geq  \frac{K C}{2} B_{n_k}^{-\frac{1}{2\eta}}   (\theta( |x_0-X_{n_k}|) +\theta( B_{n_k}^{-\frac{1}{\eta}} )),
\end{align*}
 where $\theta(h) = |h|^\alpha |\log |h||^{2\alpha}$ and where we used that $ B_{n_k}^{-\frac{1}{\eta}} <\!\!\! < |x_0-X_{n_k}|$.
 
 This shows that almost surely,  for every $x\in G'_1$, the   modulus of continuity  is larger  than $|h|^\alpha |\log |h||^{2\alpha}$.
 
 \medskip
 
 Let us now introduce the set  
 $$   {\widetilde {G_1}}  = \limsup_{j\rightarrow +\infty} \bigcup_{n\in A_j} B(X_n,B_n^{-   (1+3 {\varepsilon}_j)})   .$$

Recalling \eqref{eq3.5},     almost surely, 
$$\sum_{n\in A_j} |B(X_n,B_n^{-   (1+2 {\varepsilon}_j)}) | \leq 2^{\eta j(1+\varepsilon_j)} 2^{-\eta j  (1+3 {\varepsilon}_j)}  = j^{-2}.$$
Consequently,  $  {\widetilde {G_1}}$ has zero Lebesgue measure.

Then, a slight adaptation of the proof of Proposition \ref{prop6.1} shows that almost surely,  for every  $x_0\notin  {\widetilde {G_1}}$, there exists $K_{x_0} >0$ such that for any $x$ close to $x_0$,
$$|F(x)-F(x_0)| \leq K_{x_0} |x-x_0|^{ {\alpha} }  |\log_2 |x-x_0|\ |^{3+\alpha}  .$$
  The modification consists in replacing $\delta$ by  $1+3\ep_j$, and adapting accordingly the computations.

\medskip

The conclusion follows by considering the set $G= G'_1\setminus  {\widetilde {G_1}} $. Indeed, since $G'_1$ and ${\widetilde {G_1}} $ respectively have  full   and zero Lebesgue measure, $G $ has full Lebesgue measure. And the two arguments above show that almost surely, for every $x_0 \in G$, the modulus of continuity $\theta_{x_0}$ of $F$ at $x_0$ satisfies
 $$ |h|^\alpha |\log |h||^{2\alpha}  \leq \theta_{x_0} (h) \leq |h|^{ {\alpha} }  |\log_2 |h|\ |^{3+\alpha} ,$$
 hence items (ii) and (iii) of Theorem \ref{mainth2}.

\section{Perpectives} \label{sec:10}

The case  where $\alpha>1$ is a possible extension of our article.

\medskip

It is also a natural question for applications to ask whether the sample paths of $F$ satisfy a multifractal formalism. 

\medskip

It would  be interesting to determine whether $F$ possess chirps or oscillating singularities, i.e. locally behaves like
$$|x-x_0|^\alpha |\log|x-x_0||^\beta$$
around some points $x_0$. Chirps are a key notion in many domains - for instance, the existence of gravitational waves has been experimentally proved thanks to wavelet based-algorithms  able to detect chirps (that are the  signature of  coalescent binary black holes) in signals extracted from the LIGO and VIRGO interferometers.

\medskip

Finally, it is worth investigating the case where   the series defining $F$ does not converge uniformly, this may occur for some choices of the parameters $\alpha$ and $\eta$ (recall that in the   present paper, the uniform convergence follows from the sparse distribution of the pulses). In this situation, the relevant quantities to analyze are    the $p$-exponents of $F$ as defined in  {{\cite{ref32}}}: A function $f$ belongs to $T^p_\alpha(x_0)$ (which generalizes the spaces $C^\alpha(x_0)$)  when there exist a polynomial $P$ and a constant $C>0$ such that
$$\mbox{for every sufficiently small $h>0$, } \ \ \left(  \frac{1}{h^d} \int_{B(0,h)} |f( x)-P(x)|^p\right) ^{1/p} \leq C |h|^\alpha.$$
Then the $p$-exponent is $h^p_f(x_0) = \sup \{\alpha\geq 0: f\in T^p_\alpha(x_0)\}$, and  the multifractal analysis of the $p$-exponents of $F$ is a challenging issue.

\section*{acknowledgments}
The authors thank   St\'ephane Jaffard for enlightening discussions around this article.

\end{document}